\documentclass[10pt]{amsart}
\usepackage{verbatim}
\usepackage{tikz}
\usepackage{amssymb}
\usepackage{graphicx,color}
\usepackage{amsmath}
\usepackage{amscd} 

 \newtheorem{theorem}{Theorem}[section]
 \newtheorem{thm}[theorem]{Theorem}
 \newtheorem{lem}[theorem]{Lemma}
 \newtheorem{prop}[theorem]{Proposition}
 \newtheorem{cor}[theorem]{Corollary}

 \newtheorem{dfn}[theorem]{Definition}

 \newtheorem{rem}[theorem]{Remark}

 \numberwithin{equation}{section}

\def\wt{\widetilde}
\newcommand{\R}{{\mathbb R}}

\newcommand{\Z}{{\mathbb Z}}
\newcommand{\C}{{\mathbb C}}
\def\Sp{{\mathbb{S}}}
\def\RP{{\mathbb{RP}}}

\def\N{{\mathbb N}}

\def\Prob{{\mathbb P}}

\def\Lam{\Lambda}
\def\gam{\gamma}
\def\ov{\overline}

\def\Ok{{\mathcal O}}
\def\A{{\mathcal A}}

\def\Hk{{\mathcal H}}

\def\Jk{{\mathcal J}}

\def\Xk{{\mathcal X}}

\def\Uk{{\mathcal U}}
\def\Yk{{\mathcal Y}}

\def\be{\begin{equation}}
\def\ee{\end{equation}}
\def\bi{{\mathbf i}}
\def\bj{{\mathbf j}}

\def\bu{{\mathbf u}}
\def\bv{{\mathbf v}}
\def\bw{{\mathbf w}}
\def\half{\frac{1}{2}}
\newcommand{\Ek}{{\mathcal E}}
\newcommand{\Fk}{{\mathcal F}}

\def\wtil{\widetilde}
\newcommand{\Wk}{{\mathcal W}}
\newcommand{\eps}{{\varepsilon}}

\def\ov{\overline}

\def\trace{{\rm trace}}

\def\H\pi{H_m(p_{i\,\displaystyle{\cdot}})}
\def\ve1{\vec{1}}

\def\Ak{{\mathcal A}}

\def\lam{\lambda}
\def\Lam{\Lambda}
\def\Sig{\Sigma}
\def\sig{\sigma}

\renewcommand{\le}{\leqslant}\renewcommand{\leq}{\leqslant}
\renewcommand{\ge}{\geqslant}\renewcommand{\geq}{\geqslant}
\renewcommand{\setminus}{\smallsetminus}

\newcommand{\supp}{\mathrm{supp} \,}

\title[Diophantine property of matrices and projective IFS]
{Diophantine property of matrices and attractors of projective iterated function systems in $\RP^1$}

\author[B.\ Solomyak]{BORIS SOLOMYAK}
\author[Y.\ Takahashi]{YUKI TAKAHASHI}

\address{Yuki Takahashi, Department of Mathematics, Bar-Ilan University, Ramat Gan, Israel}
\curraddr{Advanced Institute for Materials Research, Tohoku University, Sendai, Japan}
\email{yuki.takahashi.e6@tohoku.ac.jp}

\address{Boris Solomyak, Department of Mathematics, Bar-Ilan University, Ramat Gan, Israel}
\email{bsolom3@gmail.com}

\thanks{Both authors were supported by the Israel Science Foundation grant 396/15 (PI. B.\ Solomyak).}

\date{\today}

\begin{document}

\vspace{18mm}
\setcounter{page}{1}
\thispagestyle{empty}

\begin{abstract}
We prove that almost every finite collection of matrices in $GL_d( \mathbb{R} )$ and $SL_d(\R)$ with positive entries is Diophantine. Next we restrict ourselves to the case $d=2$. A finite set of $SL_2(\R)$ matrices induces a (generalized) iterated function system on the projective line $\RP^1$. Assuming uniform hyperbolicity and the Diophantine property, we show that the dimension of the attractor equals the minimum of 1 and the critical exponent. 
\end{abstract}
\maketitle

\section{Introduction and main results}

\subsection{Diophantine property of matrices} 
Recently there has been interest in Diophantine properties in non-Abelian groups. The following is a variant of \cite[Definition 4.2]{GJS1999}.

\begin{dfn}
Let $\Ak = \{A_i\}_{i\in \Lam}$ be a finite subset of a semi-simple Lie group $G$ equipped with a metric $\varrho$. Write $A_\bi = A_{i_1}\cdots A_{i_n}$ for $\bi = i_1\ldots i_n$.
We say that the set $\Ak$ is {\em Diophantine} if there exists a constant $c>0$ such that for every $n\in \N$, we have
 \begin{equation} \label{Dioph1}
\bi,\bj\in \Lam^n,\ A_\bi\ne A_\bj \implies \varrho(A_\bi,A_\bj) > c^n.
\end{equation}
The set $\Ak$ is {\em strongly Diophantine} if there exists $c>0$ such that for all $n\in \N$,
\begin{equation} \label{Dioph2}
\bi,\bj\in \Lam^n,\ \bi\ne \bj \implies \varrho(A_\bi,A_\bj) > c^n.
\end{equation}
\end{dfn}

It is enough to check (\ref{Dioph1}) for all $n$ sufficiently large, since for a finite number of $n$'s it always holds with some constant $c>0$.
Clearly, $\Ak$ is strongly Diophantine if and only if it is Diophantine and generates a free semigroup; thus it also suffices to check (\ref{Dioph2}) for all $n$ sufficiently large.

Gamburd, Jakobson, and Sarnak \cite[Definition 4.2]{GJS1999} gave a definition of a Diophantine set, which is similar to ours, except that they considered words in the alphabet $\Ak\cup \Ak^{-1}=\A\cup \{A_i^{-1}\}_{i\in \Lam}$. Diophantine-type questions in groups arise in connection with spectral gap estimates, see \cite{GJS1999,Bourgain2014}.

See \cite{ABRS2015,ABRS2018} for a recent discussion of Diophantine properties in groups and related problems. There, the definition is more general, replacing the separation in (\ref{Dioph1}) with a negative power of the cardinality of the $n$-ball in the word metric in the group generated by $\A$. 
In \cite{ABRS2018} a semi-simple Lie group $G$ is called Diophantine, if almost every $k$  elements of $G$, chosen independently at random according to the Haar measure,  together with their inverses, form a Diophantine set in $G$. 
Gamburd et al. \cite{GJS1999} conjectured that  $SU_2(\R)$ is Diophantine. More generally, it is conjectured that semi-simple Lie groups are Diophantine. Kaloshin and Rodnianski \cite{KR2001} proved a weaker Diophantine-type property: for a.e. $(A,B) \in SO_3(\R)\times SO_3(\R)$, there exists $c>0$  such that for any $n\ge 1$ and any two distinct words $W_1, W_2$ over the set $\Ak=\{A,B,A^{-1},B^{-1}\}$ of length $n$,
$$
\|W_1-W_2\|\ge c^{n^2}. 
$$
It is stated in \cite{KR2001} that their method is general, and applies to $SU_2(\R)$ as well, and also to $m$-tuples of matrices for any $m\ge 2$. In \cite[pp.\,9-10]{ABRS2018} it is mentioned that the same method works for other semi-simple Lie groups. Breuillard \cite[Cor.1.11]{Breuillard2011} showed that a closely related weak form of Diophantine property holds for every $m$-tuple generating a dense subgroup of $SU_2(\R)$.

\medskip

Next we state our first result.
For any collection of linearly independent vectors $v_1,\ldots,v_{d}$ in $\R^{d}$ consider the simplicial cone
\begin{equation} \label{cone}
\Sig=\Sig_{v_1,\ldots,v_{d}} = \{x_1 v_1 + \cdots + x_{d} v_{d}:\ x_1,\ldots,x_{d}\ge 0\}.
\end{equation}
If a matrix $A\in GL_{d}(\R)$ satisfies
$$
A({\Sig}\setminus \{0\}) \subset \Sig^\circ,
$$
we  say that $\Sig$ is {\em strictly invariant} for $A$.
Given a  cone $\Sig=\Sig_{v_1,\ldots,v_{d}}$, denote by $\Xk_{\Sig,m}$ (respectively, $\Yk_{\Sig,m}$) the set of all $GL_{d}(\R)$ (respectively, $SL_{d}(\R)$) $m$-tuples of matrices for which $\Sig$ is strictly invariant. We consider $\Xk_{\Sig,m}$ as an open subset of $\R^{d^2m}$ and $\Yk_{\Sig,m}$ as a $(d^2-1)m$-dimensional submanifold.

\begin{thm}\label{main_thm} Let $\Sig=\Sig_{v_1,\ldots,v_{d}}$ be a simplicial cone in $\R^{d}$ and $m\ge 2$.

{\rm (i)}
For a.e. $\mathcal{A} \in \mathcal{X}_{\Sig, m}$, the $m$-tuple
$\mathcal{A}$ is strongly Diophantine. 
In particular, a.e. $m$-tuple of positive $GL_{d}(\R)$  matrices is strongly Diophantine.

{\rm (ii)}
For a.e. $\mathcal{A} \in \mathcal{Y}_{\Sig, m}$, the $m$-tuple
$\mathcal{A}$ is strongly Diophantine.
In particular, a.e. $m$-tuple of positive $SL_{d}(\R)$  matrices is strongly Diophantine.
\end{thm}

\begin{rem}{\em
1. Unfortunately, our results do not cover any example of a symmetric set, since the strict invariance property cannot hold for a matrix  $A$ and $A^{-1}$ simultaneously. 

2. Every $m$-tuple of matrices with algebraic  entries is Diophantine (but not necessarily strongly Diophantine), see, e.g., \cite[Prop.\,4.3]{GJS1999}.

3. It is well-known that Diophantine numbers in $\R$ form a set of full measure, which is, however, meagre in Baire category sense (its complement contains a dense $G_\delta$ set). Baire category genericity of non-Diophantine $m$-tuples in $SU_2(\R)$ has been pointed out in \cite{GJS1999}. In $G=SL_d(\R)$ the situation is different, since there are, for example, open sets of $m$-tuples in $G\times G$ which satisfy (\ref{Dioph2}).
For instance, if $\R^d_+$ is mapped by $A,B$ into closed cones that are disjoint, except at the origin, then (\ref{Dioph2}) holds for $\{A,B\}$. On the other hand, there are open sets in  $(SL_d(\R))^m$ in which non-Diophantine $m$-tuples are dense. 
For instance, the set of elliptic matrices in $SL_2(\R)$ is open, and 
a standard argument shows that a generic $m$-tuple that contains an elliptic matrix is not Diophantine.
}
\end{rem}

The scheme of the proof of Theorem~\ref{main_thm} is as follows. We consider the induced action of the matrices on the projective space, and show that, given a non-degenerate family of 
$m$-tuples strictly preserving an open set, depending on a parameter real-analytically, for all parameters outside an exceptional set of zero Hausdorff dimension, the induced iterated function system (IFS) satisfies a version of the  {\em exponential separation condition}. This property implies the strong Diophantine condition for the matrices. We then locally foliate the space of $m$-tuples of matrices  and apply Fubini's Theorem. The result on the zero-Hausdorff dimensional set of exceptions uses the notion of {\em order-$k$ transversality}, which is a modified version of that which appeared in the work of Hochman \cite{Hochman2014,Hochman2015}. The strict open set preservation property is needed to ensure that the induced IFS is contracting. 


\subsection{Projective IFS and linear cocycles}
Let $\Ak = \{A_i\}_{i\in \Lam}$ be a finite collection of $GL_d(\R)$ matrices. The linear action of $GL_d(\R)$ on $\R^{d}$ induces an action on the projective space $\RP^{d-1}$, and thus $\Ak$ defines an IFS $\Phi_\Ak=\{\varphi_A\}_{A\in \Ak}$ on $\RP^{d-1}$, called a {\em (real) projective IFS}. Such IFS were studied by Barnsley and Vince \cite{BV2012}, and by De Leo \cite{Leo2015a,Leo2015b}.
Following \cite{BV2012}, we say that the IFS $\Phi_\Ak$ has an attractor $K$ if  for every nonempty compact set $B$ in  a neighborhood  of $K$, we have $\lim_{k\to \infty}\Phi_\Ak^k(B) = K$ in the Hausdorff metric, where $\Phi_\Ak(B) = \bigcup_{A\in \Ak} \varphi_A(B)$. It is shown in \cite{BV2012} that the attractor is necessarily unique. Assume, for simplicity, that matrices in $\A$ are orientation-preserving, that is, $\A \subset GL^+_d(\R) = \{A\in GL_d(\R): \det(A)>0\}$. The action of $GL^+_d(\R)$ factors through the $SL_d(\R)$ action, via $A\mapsto |\det A|^{-1/d}\cdot A$; thus, it is often enough to work with families of $SL_d(\R)$ matrices. 


 An alternative, but closely related viewpoint, is to consider the linear cocycle $A: \Lam^\Z\to SL_d(\R)$ over the shift on $\Lam^\Z$, defined by $A(\bi) = A_{i_1}$, see \cite{BG2009}.
 Here we restrict ourselves to the case of $d=2$, which was  investigated in great detail by Yoccoz \cite{Yoccoz2004} and Avila, Bochi, and Yoccoz \cite{ABY2010}.
 
 \begin{dfn}[De Leo \cite{Leo2015a}] \label{def-hyperb}
 A finite set of $SL_2(\R)$ matrices $\Ak = \{A_i\}_{i\in\Lam}$ is called {\em hyperbolic} if there exist $c>0$ and $\lam>1$ such that 
 \be \label{eq-hyper5}
 \|A_\bi\|\ge c\lam^n\ \ \mbox{for all}\ \bi\in \Lam^n,\ n\in \N.
 \ee
 \end{dfn}

Recall that $A\in SL_2(\R)$ is  elliptic, parabolic, or hyperbolic if $\trace(A)$ is, respectively, in $(-2,2)$, $\{-2,2\}$,  or $\R\setminus [-2,2]$. Definition~\ref{def-hyperb} is consistent with this, in the sense that $A\in SL_2(\R)$ is hyperbolic if and only if $\{A\}$ is a hyperbolic set of matrices.
The property (\ref{eq-hyper5}) was used in \cite{Yoccoz2004,ABY2010} as a criterion (necessary and sufficient) for uniform hyperbolicity of the cocycle.
(We do not need the original definition of a uniformly hyperbolic cocycle,  referring the reader to \cite{Yoccoz2004}.)

 There is a natural identification between $[0, \pi)$ and the projective space $\RP^1$. 
Below we use this identification freely,  
and whenever necessary we view $[0, \pi)$ as $\mathbb{R} / \pi \mathbb{Z}$. For $A\in GL_2(\R)$  denote the action of $A$ on $[0, \pi)\cong \RP^1$ by the symbol $\varphi_A$. 
Denote by $d_{\Prob}$ the metric on $\RP^1$ induced from the identification with $\R/\pi\Z$. 

\begin{dfn}
A {\em multicone} is a proper nonempty open subset $U$ of $\RP^1$,  having finitely many connected components with disjoint closures.
\end{dfn}

In the following theorem we extracted the results relevant for us from \cite{ABY2010,BV2012} (note that \cite{BV2012} considers real projective IFS of any dimension).

\begin{theorem}[{\cite{ABY2010,BV2012}}] \label{th-unihyp} Let $\Ak = \{A_i\}_{i\in \Lam}$  be a family of $SL_2(\R)$ matrices and let $\Phi_\Ak$ be the associated IFS on $\RP^1$. The following are equivalent:

{\rm (i)} the IFS $\Phi_\Ak$ has an attractor $K \ne \RP^1$;

{\rm (ii)} the set of matrices $\Ak$ is hyperbolic;

{\rm (iii)} there is a multicone $U$, such that $\Phi_\Ak(\ov{U}) \subset U$;

{\rm (iv)} there is a nonempty open set $V\subset \RP^1$ such that $\Phi_\Ak$ is contractive on $\ov{V}$, with respect to a metric equivalent to $d_{\Prob}$.
\end{theorem}

By the classical Hutchinson's Theorem \cite{Hutch}, for a contractive IFS on $\ov{V}$, the attractor  $K$ is the unique non-empty compact invariant subset of $\ov{V}$. A contraction $\varphi_A$ on $\ov{V}$  has a unique fixed point, which is called the {\em attracting fixed point}. The attracting fixed points of all $\varphi_A$'s belong to the attractor $K$. Thus the attractor is not a singleton if and only if at least two of the attracting fixed points are distinct. In the latter case, as is well-known,  $K$ is perfect, i.e., it has no isolated points.

We will call a multicone $U$ satisfying  $\Phi_\Ak(\ov{U})\subset U$, a {\em strictly invariant multicone} for the family of matrices and for the IFS.  There are examples, see \cite{ABY2010}, which show that one may need  a multicone having $k$ components, for any given $k\ge 2$, even for a hyperbolic  pair of $SL_2(\R)$ matrices.

\medskip

Our next result concerns the dimension of the attractor.
Following De Leo \cite{Leo2015a}, consider the $\zeta$-function
$$
\zeta_\Ak(t) = \sum_{n\ge 1} \sum_{\bi \in \Lam^n} \|A_\bi\|^{-t},
$$
and define the {\em critical exponent} of $\Ak$ by 
\be \label{critic}
s_\Ak = \sup_{t\ge 0} \{t:\ \zeta_\Ak(t) = \infty\}.
\ee
 
\begin{thm} \label{thm-attr}
Let $\Ak=  \{A_i\}_{i\in \Lam}$ be a finite set of $SL_2(\R)$ matrices which has a strictly invariant multicone (or satisfies any of the equivalent conditions from Theorem~\ref{th-unihyp}), and let $K$ be the attractor of the associated IFS $\Phi_\Ak$ on $\RP^1$. Assume that the attractor $K$ is not a singleton.
If $\Ak$ is  strongly Diophantine, then $\dim_H(K) = \min\{1,\half s_\Ak\}$, where $s_\Ak$ is the critical exponent (\ref{critic}).
\end{thm}

In the special case when the IFS $\Phi_\Ak$ satisfies the Open Set Condition, this result is due to De Leo \cite[Th.4]{Leo2015a}.
Recall that the strong Diophantine condition holds, in particular, when $\Ak$ generates a free semigroup and all the entries of $A_i$ are algebraic. 

\begin{rem} \label{rem-DeLeo} {\em
It is further shown in \cite{Leo2015a} that for $\Ak$ hyperbolic (and in some parabolic cases),
$$
s_\Ak = \lim_{r\to \infty} \frac{\log N_{\Ak}(r)}{\log r},
$$
where $N_{\Ak}(r)$ is the number of elements of norm $\le r$ of the semigroup generated by $\Ak$. An analogy is pointed out with the classical results on Kleinian and Fuchsian groups, see, e.g., \cite{Sullivan1984}.
}
\end{rem}
 
 Let $\Phi=\Phi_\Ak$.
An alternative way to express the dimension, and one we actually use in the proof, is in terms of {\em Bowen's pressure formula}
\be \label{Bowen1}
P_\Phi(s) = 0,
\ee
where $P_\Phi(\cdot)$ is the pressure function associated with the IFS $\Phi$. 
Throughout the paper we use the notation
$$
\varphi_\bi = \varphi_{i_1}\ldots \varphi_{i_n},\ \ \mbox{where}\ \ \varphi_i = \varphi_{A_i}.
$$
The pressure is defined by 
\be \label{Bowen0}
P_\Phi(t) = \lim_{n\to \infty} \frac{1}{n} \log\sum_{\bi \in \Lam^n} \|\varphi_\bi'\|^t,
\ee
where $\|\cdot\|$ is the supremum norm on $\ov{U}$.
As will be clear from the Bounded Distortion Property,  the definition of $P_\Phi(t)$ does not depend on the choice of strictly invariant multicone $U$, and moreover,
\be \label{critic2}
2s = s_\Ak.
\ee

It is a classical result, going back to Bowen \cite{Bowen1979} and Ruelle \cite{Ruelle1982}, see also \cite{Falconer_Tech}, that if $\{\varphi_i\}_{i\in \Lam}$ is a hyperbolic IFS on $\R$ of smoothness $C^{1+\eps}$, satisfying the Open Set Condition, then the dimension of the attractor $K$ is given by Bowen's equation. In the case  that the maps $\varphi_i$ are affine, $s>0$ is the unique solution of
$$
\sum_{i\in \Lam} r_i^s = 1,
$$
where $r_i \in (0,1)$ is the contraction ratio of $\varphi_i$. For an IFS with overlaps this is not necessarily true. 
In \cite{SSU1}, Simon, Solomyak, and Urba\'nski showed that for a one-parameter family of  nonlinear IFS  with overlaps (hyperbolic and some parabolic) satisfying the {\em order-1 transversality condition}, for Lebesgue-a.e.\ parameter the dimension of the attractor is given by
\begin{equation}\label{dimul}
\dim_H (K) = \min\{1,s\},
\end{equation}
where $s$ is the solution of Bowen's equation (\ref{Bowen1}) (the solution is unique in the hyperbolic case; in the parabolic cases considered in \cite{SSU1} one needs to take the minimal solution). More recently, starting with \cite{Hochman2014}, (a version of) the following condition appeared in the literature.

\begin{dfn} \label{def-sep}
Let $\Fk = \{f_i\}_{i\in \Lam}$ be an IFS on a metric space $(X,\varrho)$, that is, $f_i:X\to X$. We say that $\Fk$ 
satisfies the \emph{exponential separation condition} on a set $X'\subseteq X$
if there exists $c > 0$  such that for all $n\in \N$ sufficiently large we have
\begin{equation}\label{exp_sep}
\sup_{x\in X'} \varrho(f_\bi(x), f_\bj(x) ) > c^{n},\ \ \mbox{for all}\  \bi,\bj\in \Lam^{n}\ \ \mbox{with}\  \ f_\bi\not \equiv f_\bj.
\end{equation}
If, in addition, the semigroup generated by $\Fk$ is free, that is, $f_\bi \equiv f_\bj \ \Longleftrightarrow\ \bi=\bj$, we say that $\Fk$ satisfies the {\em strong exponential separation condition}.
If these properties hold for infinitely many $n$, then we say that $\Fk$ satisfies the {\em (strong) exponential separation condition on $X'$ along a subsequence}. 
\end{dfn}

We prove in Proposition~\ref{prop-Dioph2} below that the (strong) exponential separation condition on a non-empty set for the projective IFS on $\RP^{d-1}$ implies the (strong) Diophantine condition for an $m$-tuple of $GL_d(\R)$ matrices. 
 
In \cite[Cor. 1.2]{Hochman2014}, Hochman proved (\ref{dimul}) 
for an affine IFS $\mathcal{F} = \{ f_i \}_{i \in \Lambda}$ on $\R$,
satisfying the strong exponential separation condition  along a subsequence on the set $X'=\{0\}$. (In fact, it follows from \cite{Hochman2014} that strong exponential separation along a subsequence on any finite subset of $\R$ implies the same conclusion.)
Thus our Theorem~\ref{thm-attr} is, in a sense, a generalization of Hochman's result to the case of contractive projective IFS.


\subsection{IFS of linear fractional transformations}

It is well-known that the action of $GL_2(\R)$ on $\RP^1$ can be expressed in terms of linear fractional transformations.
For 
\begin{equation*}
A = 
\begin{pmatrix}
a & b \\
c & d \\
\end{pmatrix}
\in GL_2(\mathbb{R}), 
\end{equation*}
let
$f_A(x) = (ax + b)/(cx + d)$,
and define $\psi : [0, \pi) \to \mathbb{R^*}$ by $\psi (\theta) = \cos\theta/\sin\theta$, where $\R^* = \R \cup \{\infty\}$.
It is easy to see that the following diagram commutes: 
\begin{center}
\usetikzlibrary{matrix}
\begin{tikzpicture}
\matrix (m) [matrix of math nodes, row sep=3em, column sep=4em, minimum width=2em]
{
[0, \pi) & {[0, \pi)} \\
\mathbb{R^*} & \mathbb{R^*} \\ };
  \path[-stealth]
    (m-1-1) edge node [left] {$\psi$} 
    (m-2-1) edge node [above] {$\varphi_A$} (m-1-2)
    (m-2-1.east|-m-2-2) edge node [above] {$f_A$}
    (m-2-2)
    (m-1-2) edge node [left] {$\psi$} (m-2-2);
\end{tikzpicture}
\end{center}
Observe that $\psi$ is smooth, and on any compact subset of $(0,\pi)$ the derivatives of $\psi$ and $\psi^{-1}$ are bounded. 
The following is then an immediate corollary of Theorem~\ref{thm-attr}, in view of Proposition~\ref{prop-Dioph2} below. 

\begin{cor}\label{cor-ifs}
Let $\mathcal{F} = \{ f_i \}_{ i \in \Lambda }$ be a finite collection of linear fractional transformations with real coefficients.  
Assume that there exists $U\subset \R$, a finite union of bounded open intervals with disjoint closures, such that 
$f_i( \overline{U} ) \subset U$ for all $i \in \Lambda$.  Let $K$ be the attractor of the IFS $\Fk$, and assume that $K$ is not a singleton.
If $\mathcal{F}$ satisfies the strong exponential separation condition on a non-empty subset of $\ov{U}$, then   
$
\dim_H (K) = \min\{ 1, s \}, 
$
where $s > 0$ is the 
unique zero of the pressure function $P_\Fk$. 
\end{cor}


\subsection{Furstenberg measure}
Let $\mathcal{A} = \{ A_i \}_{i \in \Lambda}$ be a finite collection of $SL_2(\mathbb{R})$ matrices, 
and let $p = ( p_i )_{i \in \Lambda}$ be a probability vector. 
Assume that $p_i > 0$ for all $i \in \Lambda$ (we always assume this for any probability vector). 
We consider the finitely supported probability measure $\mu$ on $SL_2(\R)$:
\be \label{def-mu}
\mu = \sum_{i\in \Lam} p_i \delta_{A_i}.
\ee
Our standing assumption is that $\mathcal{A}$ generates an unbounded and totally irreducible subgroup 
(i.e., does not preserve any finite set in $\RP^1$).
Then there exists a unique probability measure 
$\nu$ on $\RP^1$ satisfying $\mu\cdot\nu = \nu$, that is,
\begin{equation} \label{Furst1}
\nu = \sum_{i \in \Lambda} p_i A_i \nu, 
\end{equation}
where $A_i \nu$ is the push-forward of $\nu$ under the action of $A_i$, see  \cite{Furstenberg1963}. 
The measure $\nu$ is the {\em stationary measure}, or the  {\em Furstenberg measure}, for the random matrix product $A_{i_n}\cdots A_{i_1}$ where
the matrices are chosen i.i.d. from $\Ak$ according to the probability vector $p$.

The properties of the Furstenberg measure for $SL_2(\R)$ random matrix products, such as absolute continuity, singularity, Hausdorff dimension, etc., were studied by many authors, including \cite{Ledrappier1983,BL1985}. In  \cite{Pincus1994,Lyons2000,SSU2}
this investigation was linked with the study of IFS consisting of linear fractional transformations. 
The reader is referred to \cite{HS2017} for a discussion of more recent applications. 
We will recall the main result of \cite{HS2017}, since it will be the main tool in proving Theorem \ref{thm-attr}.

Let $\chi_{\mathcal{A}, p}$ be the \emph{Lyapunov exponent}, which is the almost sure value of the limit 
\begin{equation} \label{Lyap1}
\lim_{n \to \infty} \frac{1}{n} \log \|  A_{i_1 \cdots i_n } \|, 
\end{equation}
where $i_1, i_2, \cdots \in \Lambda$ is a sequence chosen  
randomly according to the probability vector $p = ( p_i )_{ i \in \Lambda}$. 
The Lyapunov exponent is usually defined as the almost sure value of the limit 
\begin{equation} \label{Lyap2}
\lim_{n \to \infty} \frac{1}{n} \log \|  A_{i_n \cdots i_1 } \|, 
\end{equation}
but it is well-known that (\ref{Lyap1}) and (\ref{Lyap2}) define the same value, since $A_{i_1\cdots i_n}$ and $A_{i_n\cdots i_1}$ have the same distribution. Under the standing assumptions, the limit exists almost surely and is positive \cite{Furstenberg1963}. 
The Hausdorff dimension of a measure $\nu$ is defined by
$$
\dim_H(\nu) = \inf\{\dim_H (E):\ \nu(E^c)=0\}.
$$
For a probability vector $p = ( p_i )_{i \in \Lambda}$, we denote the \emph{entropy} $H(p)$ by 
\begin{equation*}
H(p) = - \sum_{i \in \Lambda} p_i \log p_i. 
\end{equation*}

\begin{thm}[{\cite{HS2017}}] \label{thmHS}
Let $\mathcal{A} = \{ A_i \}_{i \in \Lambda}$ be a finite collection of $SL_2(\mathbb{R})$ matrices. 
Assume that $\mathcal{A}$ is strongly  Diophantine 
and generates an unbounded and totally irreducible subgroup.  
Let $p = ( p_i )_{i \in \Lambda}$ be a probability vector, 
and let $\nu$ be the associated Furstenberg measure. 
Then we have 
\begin{equation} \label{eq-Fursten}
\dim_H (\nu) = \min \Big\{ 1, \frac{H( p )}{ 2 \chi_{ \mathcal{A}, p } } \Big\}. 
\end{equation}
\end{thm}

Theorem~\ref{main_thm} implies, in particular, that the dimension formula (\ref{eq-Fursten}) holds for the Furstenberg measure associated with a.e.\ finite family of positive matrices (independent of the probability vector).

\medskip

Next  we address the question: what is the Hausdorff dimension of the support of the Furstenberg measure? Sometimes, the support is all of $\RP^1$, in which case the answer is trivially one.  The definition (\ref{Furst1}) implies that the support is invariant under the IFS $\Phi$ induced by $\Ak$. 

\begin{cor}
Let $\Ak=  \{A_i\}_{i\in \Lam}$ be a strongly Diophantine set of $SL_2(\R)$ matrices which has a strictly invariant multicone and generates a totally irreducible subgroup. Let $\mu$ be a finitely supported measure on $\Ak$ defined by (\ref{def-mu}), and $\nu$ the associated Furstenberg measure. Then
$\dim_H(\supp\nu) = \min\{1,\half s_\Ak\}$, where $s_\Ak$ is the critical exponent of $\Ak$.
\end{cor}

The corollary immediately follows from Theorem~\ref{thm-attr}. Indeed, having a strictly invariant multicone implies that the subgroup generated by $\Ak$ is unbounded, and together with total irreducibility this implies that the stationary measure $\nu$ is unique.
Its support is a compact set, invariant for the IFS $\Phi_\Ak$, hence it is the attractor, which is unique under our assumption. 
 The attractor is not a singleton by total irreducibility.

\medskip

Denote by $\Hk_m$ the set of $m$-tuples in $SL_2(\R)$ which have a strictly invariant multicone.
Proposition 6, attributed to a personal communication from Avila, which appeared, with a proof, in the paper by Yoccoz \cite{Yoccoz2004}, asserts that the interior of the complement of $\Hk_m$ in $(SL_2(\R))^m$ is $\Ek_m$,  the set of $m$-tuples which generate a semigroup containing an elliptic matrix. Observe that if an elliptic matrix is conjugate to an irrational rotation, then certainly the invariant set (support of the Furstenberg measure) is all of $\RP^1$. On the other hand, if it is conjugate to a rational rotation, then the semigroup generated by $\Ak$ contains the identity and the  strong Diophantine property fails.
We expect that our methods can be extended to cover strongly Diophantine families on the boundary of $\Hk_m$, which include parabolic systems.

\subsection{Structure of the paper}
The rest of the paper is organized as follows. In the next section we  prove Theorem \ref{main_thm}. 
In Section 3 we consider projective IFS and prove
Theorem~\ref{thm-attr}. 
 Finally, in Section 4 we include proofs of some standard technical results for the reader's convenience.  


\section{Diophantine property of $GL_{d+1}(\mathbb{R})$ and $SL_{d+1}(\mathbb{R})$ matrices}

For notational reasons it is convenient to consider $GL_{d+1}(\R)$ instead of $GL_d(\R)$.

\subsection{$GL_{d+1}(\R)$ actions}

\begin{sloppypar}
Let $A\in GL_{d+1}(\R)$ be a matrix that strictly preserves a cone $\Sig=\Sig_{v_1,\ldots,v_{d+1}}\subset \R^{d+1}$. Without loss of generality, by a change of coordinates, we can assume  that $\Sig\setminus \{0\}$ is contained in the halfspace
$\{x\in \R^{d+1}: x_{d+1}>0\}$. It is convenient to represent the induced action of $A$ on $\RP^{d}$ on the affine hyperplane $\{x\in \R^{d+1}: x_{d+1}=1\}$, and consider the corresponding action on $\R^{d}$.
To be precise, for $x= (x_1,\ldots,x_{d})\in \R^{d}$, we consider $(x,1) =  (x_1,\ldots,x_{d},1)\in \R^{d+1}$ and let
$$
f_A(x) = {\rm P}_{d}\Bigl(\frac{A(x,1)}{A(x,1)_{d+1}}\Bigr),\ \ \mbox{when}\ \ A(x,1)_{d+1}\ne 0,
$$
where ${\rm P}_{d}$ is the projection onto the  first $d$ coordinates. The components of $f_A$ are rational functions, which are, of course, real-analytic on their domain.
Consider 
$$
\ov{V} := {\rm P}_{d}(\Sig \cap \{x\in \R^{d+1}: x_{d+1}=1\}).
$$
By assumption, $f_A$ is well-defined on $\ov{V}$, and we have $f_A(\ov{V})\subset V$. 
\end{sloppypar}

We will also consider the action of $A$ on the unit sphere, given by
$$
\varphi_A(x):=A\cdot x =\frac{Ax}{\|A x\|},
$$
for a unit vector $x\in \Sp^{d}$. Consider $\ov{U}$, the intersection of $\Sig$ with the upper hemisphere. We have $\varphi_A(\ov{U}) \subset U$. Lines through the origin provide a 1-to-1 correspondence between $\ov{U}$ and $\ov{V}$, which is bi-Lipschitz in view of the assumption $\Sig \setminus \{0\} \subset \{x\in \R^{d+1}: x_{d+1}>0\}$.

It is well-known \cite{Birk1957} (see also \cite[Section 9]{BV2012}) that strictly preserving a cone implies that $\varphi_A$ is a strict contraction in the Hilbert metric on $\ov{U}$, which is by-Lipschitz with the round metric. 
We thus obtain the following:

\begin{lem} \label{lem-Hilbert}
Suppose that the finite family $\Ak = \{A_i\}_{i\in \Lam} \subset GL_d(\R)$ strictly preserves a simplicial cone $\Sig =\Sig_{v_1,\ldots,v_{d+1}}\subset \{x\in \R^{d+1}: x_{d+1}>0\}\cup \{0\}$. Then the associated IFS
$\Fk_\Ak = \{f_A\}_{A\in \Ak}$ is real-analytic and uniformly hyperbolic on $\ov{V}\subset \R^{d}$, in the sense that there exist $C>0$ and $\gam\in (0,1)$ such that 
$$
\max_{x\in \ov{V}} \|f'_\bi(x)\| \le C\gam^n,\ \ \mbox{for all}\ \bi\in \Lam^n,
$$
where $f_\bi = f_{A_{i_1}}\circ \cdots \circ f_{A_{i_n}}$ and $\|f_\bi'(x)\|$ is the operator norm of the differential at the point $x$.
\end{lem}

\subsection{From exponential separation to the Diophantine property}

Recall  the strong exponential separation condition (Definition~\ref{def-sep}). For technical reasons it is convenient to weaken it slightly.

\begin{dfn} \label{def-sep2}
Let $\Fk = \{f_i\}_{i\in \Lam}$ be an IFS on a metric space $(X,\varrho)$. We say that $\Fk$ 
satisfies the \emph{SESDC condition} on $X'\subseteq X$
if there exists  $c > 0$  such that for all $n\in \N$ sufficiently large we have
\begin{equation}\label{exp_sep2}
\sup_{x\in X'}\varrho(f_\bi(x), f_\bj(x) ) > c^{n},\ \ \mbox{for all}\  \bi,\bj\in \Lam^{n}\ \ \mbox{with}\  i_1\ne j_1.
\end{equation}
The abbreviation ``SESDC" stands for ``strong exponential separation on distinct (first order) cylinders''. 
\end{dfn}

\begin{rem} {\em
For an IFS $\{f_i\}_{i\in \Lam}$ on an interval $I\subset \R$, such that 
$$\inf_{x\in I, i\in \Lam}|f_i'(x)|\ge r_{\min}>0,$$ with $r_{\min}\in (0,1)$,
 requiring $i_1\ne j_1$ in (\ref{exp_sep2}) does not weaken the exponential separation condition (\ref{exp_sep}); 
 it only affects the constant $c$. This follows from the estimate
$$
|f_\bi(x) -f_\bj(x)| =|f_{(\bi\wedge \bj)\bu}(x) - f_{(\bi\wedge \bj)\bv}(x) |\ge r_{\min}^n |f_\bu(x) - f_\bv(x)|, \ \ \bi,\bj\in \Lam^n,
$$
where $\bi\wedge \bj$ is the common initial segment of $\bi$ and $\bj$, so that $u_1 \ne v_1$. However, in higher dimensions it is sometimes easier to check the SESDC than the strong exponential separation condition.
}
\end{rem}

\begin{prop} \label{prop-Dioph2} 
Let $\Ak$ be a finite family of $GL_{d+1}(\R)$ matrices, and let $\Phi_\Ak$ be the induced IFS on $\Sp^d$.  If $\Phi_\Ak$ satisfies the SESDC condition on a non-empty set, then $\Ak$ is strongly Diophantine.
\end{prop}

\begin{proof}
Let $C_1 = \max_{i\in \Lam}\{1,\|A_i\|\}$ and $C_2 = \max_{i\in \Lam} \{1,\|A_i^{-1}\|\}$. Suppose that $\bi\ne \bj$ in $\Lam^n$. Let us write 
$$\bi = (\bi\wedge \bj) \bu,\ \ \bj = (\bi\wedge \bj) \bv,
$$
where $\bi\wedge \bj$ is the common initial segment of $\bi$ and $\bj$, so that $\bu = u_1\ldots u_k,\ \bv = v_1\ldots v_k$ for some $k\le n$, with $u_1 \ne v_1$.
We have
\begin{equation}\label{tiu1}
\|A_\bi- A_\bj\| \ge \|A_{\bi\wedge\bj}^{-1}\|^{-1} \|A_\bu - A_\bv\| \ge C_2^{-n} \|A_\bu - A_\bv\|.
\end{equation}

\begin{lem} \label{lem-claim}
For any $A, B \in GL_{d+1}(\mathbb{R})$ and any unit vector $x\in \R^{d+1}$, we have 
\begin{equation*}
\| A \cdot x - B \cdot x \| \leq \| A^{-1} \| \bigl( 1 + \| B \| \|B^{-1}\|\bigr)\cdot \| A - B \|. 
\end{equation*}
\end{lem}

\begin{proof}
We have 
\begin{equation*}
\begin{aligned}
\| A \cdot x - B \cdot x \| &= \left\| \frac{Ax}{\| Ax \|} - \frac{Bx}{\| Bx \|} \right\| \\
&\leq \left\| \frac{Ax}{ \| Ax \|} - \frac{ Bx }{ \| Ax \| } \right\| + 
\left\| \frac{Bx}{ \| Ax \| } - \frac{ Bx }{ \| Bx \| } \right\| =: R_1 + R_2. 
\end{aligned}
\end{equation*}
Since 
\begin{equation*}
1 = \| A^{-1} (Ax) \| \leq \| A^{-1} \| \| Ax \|,   
\end{equation*}
we have $\| Ax \|^{-1} \leq \| A ^{-1}\|$. Therefore, 
\begin{equation*}
R_1 \leq \| A - B \| \cdot \| A^{-1} \|.
\end{equation*} 
Similarly, 
\begin{equation*}
\begin{aligned}
R_2 &\leq \| B \| \cdot \big| \| Ax \| - \| Bx \| \big| \cdot \| Ax \|^{-1} \| Bx \|^{-1} \\
&\leq \| B \| \cdot \| A - B \| \cdot \| A^{-1} \| \| B^{-1} \|, \\
\end{aligned}
\end{equation*}
and the desired estimate follows. 
\end{proof}

Applying the lemma to $A_\bu$ and $A_\bv$ yields, in view of $\|A_\bw\|\le C_1^n$, $\|A_\bw^{-1}\|\le C_2^n$ for any $\bw \in \Lam^k$, $k\le n$:
\begin{equation} \label{eq-tiu2}
\|A_\bu - A_\bv\| \ge 2^{-n} C_1^{-n} C_2^{-2n} \|A_\bu\cdot x- A_\bv \cdot x\|.
\end{equation}

Now we continue with the proof of the proposition. By assumption,  $\Phi_\Ak$ satisfies the SESDC condition on a non-empty set. Let $c\in (0,1)$ be the constant from the definition (\ref{exp_sep2}). It follows that for all $n\ge n_0$ there exists $x\in \Sp^d$ such that 
$$
\|A_\bu\cdot x - A_\bv\cdot x\|\ge c^k\ge c^n
$$
for all $n\ge n_0$.
Combining this inequality with (\ref{eq-tiu2}) and (\ref{tiu1}) yields
$$
\|A_\bi - A_\bj\|\ge 2^{-n} C_1^{-n} C_2^{-3n}c^n,\ \ n\ge n_0,
$$
confirming the strong Diophantine property.
\end{proof}

\subsection{Dimension of exceptions for one-parameter real-analytic families}

We consider a one-parameter family of real-analytic IFS on a compact subset of $\R^d$, and 
show that under some mild assumptions it satisfies the  SESDC condition on a single point, for parameters
outside of a Hausdorff dimension zero set. This section is based on \cite[Section 5.4]{Hochman2014}, 
 but we had to make a substantial number of modifications in the definitions and proofs.

Let $\mathcal{J}$ be a compact interval in $\R$ and $V$ a bounded open set in $\R^d$. Let $\Lam$ be a finite set, $|\Lam|\ge 2$, and suppose that for each $i \in \Lambda$ we are given a real-analytic function
$$
f_i:\,\ov{V}\times \Jk\to V. 
$$
This means that it is real-analytic on some neighborhood of $\ov{V}\times \Jk$. We will sometimes write this function as 
$$
f_{i,t}(x) = f_i(x,t),\ \ x\in \ov{V},\ \ t\in \Jk.
$$
Denote $\mathcal{F}_t = \{ f_{i, t} \}_{ i \in \Lambda }$. This is a real-analytic IFS on $\ov{V}$, depending on the parameter $t\in \Jk$ real-analytically.
For $\bi = i_1\ldots i_n$ we write $f_{\bi,t} = f_{i_1,t} \circ \cdots \circ f_{i_n,t}$.


Further, assume that this family of IFS is uniformly hyperbolic in the following  sense:  there exist $C>0$ and $0 < \gam<1$,  such that
\be \label{gluk1}
\|f'_{\bi,t}(x)\| \le C\gam^n,\ \ \ \mbox{for all}\ \ \bi\in \Lam^n, \ x\in \ov{V},\ t\in \Jk.
\ee
Here in the left-hand side is the norm of differential with respect to $x\in \R^d$.
Fix $x_0 \in V$. 
For any finite sequence $\bi \in\Lambda^n$ we define 
$$
F_\bi(t) = f_{\bi,t}(x_0).
$$
Of course, this depends on $x_0$, but we suppress it from notation. 
For $\bi\in \Lam^\N$ we have 
\begin{equation} \label{conver}
\Pi_t(\bi) = F_\bi(t) :=  \lim_{n\to \infty} F_{\bi|_n}(t), 
\end{equation}
where $\Pi_t:\Lam^\N\to \R^d$ is the natural projection corresponding to the IFS $\Fk_t$ and
$\bi|_n = i_1\ldots i_n$.  Notice that this limit is well-defined, independent of $x_0$, and is uniform in $t\in \Jk$, by  uniform hyperbolicity (\ref{gluk1}).

\medskip

\begin{lem} \label{lem-analytic}
The function $F_\bi(\cdot)$ is real-analytic on $\Jk$, for any $\bi\in \Lam^\N$. Moreover, $F_{\bi|_n}(\cdot) \to F_\bi(\cdot)$ uniformly on $\Jk$ for all $\bi\in \Lam^\N$, together with derivatives of all orders.
\end{lem}

\begin{proof}
By assumption,  for every $\bi \in \Lam^n$, the function $F_\bi$ extends to a holomorphic function in a complex neighborhood of $\Jk$, and we are going to prove that for all $\bi \in \Lam^\N$ the sequence $F_{\bi|_n}$ converges to $F_\bi$ on a sufficiently small neighborhood uniformly. In order to achieve this, note that since $f_i(x,t):\,\ov{V}\times \Jk \to V$ is real-analytic,   it can be extended to a holomorphic (complex-analytic) function $\wtil{f}_i(z,\tau)$, defined on a 
neighborhood of $\ov{V}\times \Jk$ in  $\C^d\times \C$. Denote by ${[\ov{V}]^\delta}$ the closed $\delta$-neighborhood of $\ov{V}$ in $\C^d$ and let $\wtil{f}_{i,\tau} = \wtil{f}_i(\cdot,\tau)$. 
Choose $\ell \in\N$  so that $C\gam^\ell < 1/2$. Then $\|f'_{\bi,t}(x)\|<1/2$ for $\bi\in \Lam^\ell$, and $x\in \ov{V}$. 
By continuity, there exists $\delta>0$ such that
$\wtil{f}_{\bi,t}$, with $\bi\in \Lam^\ell$, is holomorphic on $[\ov{V}]^\delta$ and
\begin{equation}\label{niki1}
\|{\wtil{f}_{\bi,t}}'(z)\|< 1/2,\ \mbox{for all}\ \bi\in \Lam^\ell,\ z\in [\ov{V}]^\delta,\ t\in \Jk.
\end{equation}
Here in the left-hand side is the norm of the differential with respect to $z\in \C^d$. Thus, each $\wtil{f}_{\bi,t}$, with $\bi\in \Lam^\ell$, is a strict contraction on $[\ov{V}]^\delta$,  and since $\wtil{f}_{\bi,t}({\ov{V}}) = f_{\bi,t}(\ov{V})\subset V$, we obtain that $[\ov{V}]^\delta$ is mapped into its interior by $\wtil{f}_{\bi,t}$, for $t\in \Jk$. Then the same must be true for all $\tau$ in a sufficiently small complex neighborhood of $\Jk$, which we denote by $\Ok$. 
We can find a constant $L>1$ such that 
\begin{equation} \label{niki2}
\|\wtil{f}_{\bj,\tau}'(z)\|\le L,\ \mbox{for all}\ \bj \ \mbox{such that}\ \ |\bj| \le \ell-1,\ z\in [\ov{V}]^\delta,\ \tau\in \Ok,
\end{equation}
since there are finitely many holomorphic functions  involved.

Now, it follows from (\ref{niki1}) and (\ref{niki2}) that the function $\wtil{f}_{\bj,t}$, for {\em all}\ \  $\bj\in \bigcup_{n=1}^\infty \Lam^n$ and $t\in \Jk$, is well-defined and holomorphic in the interior of $\Wk:= [\ov{V}]^{\delta/L}$, and moreover, it maps $\Wk$ into $[\ov V]^\delta$. In addition, $\wtil{\Fk}^\ell_\tau = \{\wtil{f}_{\bi,\tau}\}_{\bi \in \Lam^\ell}$ is a strictly contracting IFS on $\Wk$, depending on
$\tau\in \Ok$ holomorphically. 
It follows that the finite iterates $\tau\mapsto \wtil{f}_{\bi|_{n},\tau}(x_0)$ converge  to $\Pi_\tau(\bi)$, the natural projection for $\wtil{\Fk}_\tau$, as $n\to \infty$, uniformly for $\tau\in \Ok$.
The uniform limit of holomorphic functions in an open set in $\C$ is holomorphic, and since $F_\bi(t) = \Pi_t(\bi)$ is the restriction of a holomorphic map to an interval on the real line, it is real-analytic. The uniform convergence of holomorphic functions implies uniform convergence of their derivatives as well.
\end{proof}

Next, for $\bi,\bj \in \bigcup_{n=1}^\infty \Lam^n$, let
\begin{equation*}
\Delta_{\bi, \bj}(t) = F_\bi(t) - F_\bj(t):\,\Jk \to \R^d.
\end{equation*}
Recall that this depends on $x_0$: $\Delta_{\bi, \bj}(t) =  f_{\bi,t}(x_0) - f_{\bj,t}(x_0)$.
For any $\eps > 0$, let
\begin{equation*}
E_{\eps} = \bigcap_{N=1}^{\infty} \bigcup_{n > N} 
\Big( \bigcup_{\bi, \bj \in \Lambda^n, i_1 \neq j_1}  \, ( \Delta_{\bi, \bj} )^{-1} B_{\eps^n} \Big)
\end{equation*}
and define the exceptional set $E=E(x_0)$ by
\begin{equation} \label{def-E}
E = \bigcap_{\eps > 0} E_{\eps}, 
\end{equation}
where $B_{\eps^n} = \{x\in \R^d:\ \|x\| \le \eps^n\}$ is the Euclidean ball in $\R^d$.

\begin{lem} \label{lem-ESDC}
If $t \notin E=E(x_0)$, then 
$\mathcal{F}_t$ satisfies the SESDC condition on $\{x_0\}\subset \ov V$.
\end{lem}

\begin{proof} Observe that $t \notin E$ implies $t \notin E_\eps$ for some $\eps>0$, hence $|f_{\bi,t}(x_0) - f_{\bj,t}(x_0)|\ge \eps^n$, for all $\bi,\bj\in \Lam^n$ with $i_1\ne j_1$, for all $n$ sufficiently large. \end{proof}

\begin{rem} {\em
In \cite{Hochman2014,Hochman2015} Hochman considered the case where $\mathcal{F}_t$ is an affine IFS. 
He defined the sets $E'_{\eps}$ and $E'$ as follows:
\begin{equation*}
E'_{\eps} = \bigcup_{N=1}^{\infty} \bigcap_{n > N} 
\Big( \bigcup_{\bi, \bj \in \Lambda^n, \bi \neq \bj} ( \Delta_{\bi, \bj} )^{-1} B_{\eps^n} \Big)
\end{equation*}
and 
\begin{equation} \label{def-Etag}
E' = \bigcap_{\eps > 0} E'_{\eps}. 
\end{equation}
If $t \notin E'$ then $\mathcal{F}_t$ satisfies the strong exponential separation condition on $\{x_0\}$      {\em along a subsequence}.
}
\end{rem}


  Denote by $[w]_q$ 
the $q$-th component of a vector $w\in \R^d$, $1\le q \le d$.

\begin{dfn}
 We say that the family of IFS $\Fk_t,\ t\in \Jk$, as above, is {\em non-degenerate in $q$-th component} for some $1\le q \le d$ if
\be\label{nondegen2}
[\Delta_{\bi,\bj}(\cdot)]_q\not\equiv 0 \  \ \mbox{for all $\bi, \bj \in \Lambda^{\mathbb{N}}$ with $i_1\ne j_1$.}
\ee
\end{dfn}

In \cite{Hochman2014} an IFS family on $\R$ is called non-degenerate if $\Delta_{\bi,\bj}\equiv 0  \iff \bi=\bj$. This condition is equivalent to (\ref{nondegen2}) for $d=1$ when the maps of the IFS are injective. 

We next prove the following: 

\begin{thm}\label{thm_important}
Suppose that the  family of IFS $\Fk_t,\ t\in \Jk$, is non-degenerate in $q$-th component for some $1\le q\le d$.
Then the set $E$ from (\ref{def-E}) has Hausdorff dimension zero, and therefore, $\Fk_t$ satisfies the  SESDC condition on $\{x_0\}$ for  all parameters $t\in \Jk$ outside of an exceptional set of Hausdorff dimension zero.
\end{thm}

Hochman \cite{Hochman2014,Hochman2015} proved, for a non-degenerate family of affine IFS, with a real-analytic dependence on parameter, that the set $E'$ from (\ref{def-Etag}) has zero {\em packing dimension}.

\begin{cor}\label{cor_sep}
For a family of IFS $\Fk_t,\ t\in \Jk$, as above,
assume that there exist $t_0 \in \mathcal{J}$ and $q$, $1\le q \le d,$
such that the sets $\pi_q\left( f_{i, t_0}( \overline{V} )\right)$ are pairwise disjoint for $i\in \Lam$, where $\pi_q:\R^d\to \R$ is the orthogonal projection to the $q$-th coordinate.
Then (\ref{nondegen2}) holds, and hence the set $E$ from (\ref{def-E}) has Hausdorff dimension zero. 
\end{cor}

To deduce the corollary from Theorem~\ref{thm_important}, it suffices to note that $[F_\bi(t_0)]_q \in \pi_q\left(f_{i_1,t_0}(\ov V)\right)$ whence $\Delta_{\bi,\bj}(t_0)\ne 0$ whenever $i_1\ne j_1$.

\medskip

For any smooth function $F : \mathcal{J} \to \R^d$,  denote $F^{(p)}(t) = \frac{d^p}{dt^p} F(t)$. 

\begin{dfn}
The family $\{ \mathcal{F}_t \}_{t \in \mathcal{J} }$ 
is said to be \emph{transverse of order $k$ in $q$-th component} for some $q, \ 1\le q \le d$, if there exists $c > 0$ such that for all $n \in \mathbb{N}$ 
and $\bi, \bj \in \Lambda^{n}$, with $i_1 \neq j_1$, we have 
\begin{equation*}
 \forall t \in \mathcal{J}, \ \exists p \in \{ 0, \cdots, k \} \text{\, s.t.  } 
 \left|\left[ \Delta^{(p)}_{\bi, \bj}(t)\right]_q \right| > c. 
\end{equation*}  
\end{dfn}


\begin{lem}\label{delta_prop}
Suppose that the non-degeneracy in $q$-th component (\ref{nondegen2}) holds.
Then $\{ \mathcal{F}_t \}_{t \in \mathcal{J}}$ is transverse of order $k$ in $q$-th component for some $k \in \mathbb{N}$. 
\end{lem}

\begin{proof}
Suppose that for all $k\in \N$ the family $\{ \mathcal{F}_t \}_{t \in \mathcal{J}}$ is not transverse of order $k$ in $q$-th component . 
Then by assumption, for $\{c_k\}$ with $c_k < 1/k$, 
we can choose $n(k)$, $\bi^{(k)}, \bj^{(k)} \in \Lambda^{n(k)}$, 
with $i^{(k)}_1 \neq j^{(k)}_1$ and a point $t_k \in \mathcal{J}$ 
such that 
\begin{equation*}
\left| \left[\Delta^{(p)}_{ \bi^{(k)}, \bj^{(k)} } (t_k)\right]_{q} \right| < c_k 
\end{equation*}
for $0 \leq p \leq k$. Since $\Lam$ is finite,
passing to a subsequence $\{ k_{l} \}$, we can assume that $t_{k_{l}} \to t_0\in \Jk$, 
$\bi^{( k_{l} )} \to \bi \in \Lambda^{\mathbb{N}}$ and $\bj^{( k_{l} )} \to \bj \in \Lambda^{\mathbb{N}}$, with $i_1\ne j_1$. By Lemma~\ref{lem-analytic}, the complex extension of 
$\Delta_{\bi^{( k_{l} )}, \bj^{( k_{l} )}}$ converges to the complex extension of $\Delta_{\bi, \bj}$ uniformly on a complex neighborhood of $\Jk$, and hence the same holds for $p$-th derivatives.
Thus for all $p \geq 0$, we have
\begin{equation*}
\left[ \Delta^{(p)}_{\bi, \bj} (t_0) \right]_q = 
\lim_{l \to \infty} 
\left[ \Delta^{(p)}_{ \bi^{(k_{l}),} \bj^{(k_{l})} } (t_{k_{l}}) \right]_q = 0. 
\end{equation*}
Since $[\Delta_{\bi, \bj}]_q$ is real-analytic, 
the vanishing of its derivatives implies $[\Delta_{\bi, \bj}]_q \equiv 0$ on $\Jk$, contradicting (\ref{nondegen2}), since $i_1\ne j_1$ by construction.
\end{proof}

For a $C^k$-smooth function $F : J \to \mathbb{C}$, where $J\subset \R$ is a compact interval, write 
\begin{equation*}
\| F \|_{ J, k } = \max_{ p \in \{0, \cdots, k\} } \sup_{t \in J} | F^{(p)}(t) |,\ \ \|F\|_{J} = \|F\|_{J,0},
\end{equation*}
and similarly for vector-functions.

\begin{lem}[Lemma 5.8 in \cite{Hochman2014}]\label{key_lem}
Let $k \in \N$ and let $F : J\to \R$ be a $k$ times continuously differentiable function on a compact interval $J\subset \R$.
Let $M = \| F \|_{J, k}$, and let $0 < b < 1$ 
be such that for every $t \in J$ there is  $p \in \{ 0, \cdots, k\}$ with 
$| F^{(p)}(t) | > b$. Then there exists a constant $C=C_{b,M,|J|}\ge 1$ such that for every $0 < \rho < (b/2)^{2^k}$, 
the set $F^{-1} (-\rho, \rho)\subset J$ can be covered by 
$C^k$ intervals of length 
$\leq 2( \rho/b )^{1/2^k} $ each. 
\end{lem}

\begin{lem} \label{lem-conc}
If the family of IFS $\{ \mathcal{F}_t \}_{t \in \mathcal{J} }$ is transverse 
of order $k \geq 1$ in $q$-th component, for some $q$, $1\le q\le d$, on the compact interval $\mathcal{J}$, 
then the set $E$ from (\ref{def-E}) has Hausdorff dimension zero. 
\end{lem}

\begin{proof}
Extending the real-analytic functions to the complex plane, as in Lemma~\ref{lem-analytic}, since 
\begin{equation*}
\sup_n \sup_{\bi, \bj \in \Lambda^n, i_1 \neq j_1} \| \Delta_{\bi, \bj} \|_{\Ok} < \infty
\end{equation*}
on a neighborhood $\Ok$ of $\Jk$,
and $\Delta_{\bi, \bj}(\cdot)$ is holomorphic on $\Ok$ for all $\bi, \bj \in \Lambda^n$, we have 
\begin{equation} \label{normal}
M:=\sup_{n} \sup_{\bi, \bj \in \Lambda^n, i_1 \neq j_1} \| \Delta_{\bi, \bj} \|_{ \mathcal{J} , k} < \infty.  
\end{equation}
Let 
\begin{equation} \label{En}
E_{\eps, n} = \bigcup_{\bi, \bj \in \Lambda^n, i_1 \neq j_1} ( \Delta_{\bi, \bj} )^{-1}( B_{\eps^n} ). 
\end{equation}
Then 
\begin{equation}\label{E}
E_{\eps} = \bigcap_{N = 1}^{\infty} \bigcup_{n > N} E_{\eps, n}. 
\end{equation}
Let $\bi, \bj \in \Lambda^n$, with $i_1\ne j_1$. 
Since the family is transverse in $q$-th component, we can apply Lemma  \ref{key_lem} to  $[\Delta_{\bi,\bj}]_q$, to obtain that for $\eps$ sufficiently small, the set 
\begin{equation*}
( \Delta_{\bi, \bj} )^{-1} ( B_{\eps^n} ) \subseteq [\Delta_{\bi,\bj}]_q^{-1}\bigl([-\eps^n,\eps^n]\bigr)
\end{equation*}
may be covered by $C^k$ intervals of length $\leq 2( \eps^n \cdot c^{-1} )^{1/2^k}$.
It follows that the set  $E_{\eps,n}$ from (\ref{En})
may be covered by $O( | \Lambda |^{2n} \cdot C^k )$ 
intervals of length $\leq 2 ( \eps^n \cdot c^{-1} )^{1/2^k}$.
Fix $s>0$ and choose $\eps>0$ such that $|\Lam|^2 \eps^{s/2^k}<1$. Writing $\Hk^s$ for the $s$-dimensional Hausdorff measure, we obtain from (\ref{E}) that
$$
\Hk^s(E_\eps) \le O(1)\cdot \sum_{n\ge 1} |\Lam|^{2n} C^k {\bigl(\eps^n \cdot c^{-1}\bigr)}^{s/2^k} < \infty.
$$
 It follows that $\Hk^s(E_\eps)<\infty$ whence $\dim_H(E) \le \dim_H(E_\eps) \le s$, and since $s>0$ was arbitrary we obtain $\dim_H(E)=0$.
\end{proof}

\begin{proof}[Proof of Theorem~\ref{thm_important}]
This is now immediate from Lemmas \ref{delta_prop} and \ref{lem-conc}.
\end{proof}


\subsection{Proof of Theorem \ref{main_thm}}

The next lemma follows by an application of Fubini's Theorem.
 
\begin{lem}\label{Fubini}
Let $F \subset \mathbb{R}^n$ and let $v \in \mathbb{R}^n$ be a nonzero vector. 
Assume that for every $x_0 \in \mathbb{R}^n$, 
the set $\{ x_0 + t v : t \in \mathbb{R} \} \cap F$ 
has 1-dimensional Lebesgue measure $0$. Then the set $F$ has $n$-dimensional Lebesgue measure $0$. 
\end{lem}


\begin{proof}[Proof of Theorem \ref{main_thm}] (i) Let $\Sig=\Sig_{v_1,\ldots,v_{d+1}}$ be a simplicial cone in $\R^{d+1}$. By a coordinate change, we can assume without loss of generality, that $\Sig\setminus \{0\}$ is contained in the subspace $x_{d+1}>0$.
Let $\Uk \subset \Xk_{\Sig,m}$ be a small open set in $(GL_{d+1}(\R))^m$ of 
$m$-tuples of matrices for which $\Sig$ is strictly invariant. Choose
vectors $w_i \in \R^{d+1} \ (i \in \Lambda)$, such that $[w_i]_{d+1} = 1$ for all $i$ and for some fixed $q$, $1\le q \le d$, the components $[w_i]_q$, for $i\in \Lam$, are all distinct, in such a way that 
\be \label{w-cond}
w_i \in \Sigma_{ A_i v_1, \ldots,A_i v_{d+1} }\ \ \mbox{ for all } \ \ (A_i)_{i \in \Lambda} \in \Uk.
\ee
This is possible when $\Uk$ is sufficiently small. (In fact, there is no difficulty in ensuring that all components of $w_i$ are distinct since the cones have nonempty interior.)
Let $(A_i)_{i \in \Lambda} \in \Uk$, and 
for each $t \geq 0$ and $i\in \Lam$ let $A_{i,t}$ be such that 
$$A_{i, t} v_j = A_i v_j + t w_i,\ j=1,\ldots,d+1.$$ 
Condition (\ref{w-cond}) guarantees that $\{A_{i,t}v_j\}_{j=1}^{d+1}$ is linearly independent, and hence $A_{i,t}\in GL_{d+1}(\R)$ for all $t\ge 0$. This is a consequence of the following elementary claim.

\medskip

\noindent
{\bf Claim.} {\em Let $y_1,\ldots,y_{d+1}\in \R^{d+1}$ be linearly independent, and suppose that $w = \sum_{k=1}^{d+1} a_k y_k$ for some $a_k\ge 0$. Then
the family $\{y_1 + w,\ldots, y_{d+1} + w\}$ is linearly independent as well.}

\begin{proof}[Proof of the Claim.]
We have
$$
\sum_{j=1}^{d+1} c_j \Bigl(y_j + \sum_{k=1}^{d+1} a_k y_k\Bigr) = 0\ \ \Longrightarrow \ \ \sum_{j=1}^{d+1} \Bigl(c_j + a_j \sum_{k=1}^{d+1} c_k \Bigr) y_j = 0,
$$
hence $c_j + a_j \sum_{k=1}^{d+1} c_k = 0$ for all $j$. If $\sum_{k=1}^{d+1} c_k\ne 0$, we obtain a contradiction, in view of $a_j\ge 0,\ j=1,\ldots, d+1$; thus $c_j=0,\ j=1,\ldots, d+1$, as claimed.
\end{proof}

Let $\Ak_t = \{A_{i,t}\}_{i\in \Lam}$ be the family of matrices defined above, for $t\ge 0$, and let $\Fk_t=\Fk_{\Ak_t}$ be the corresponding one-parameter family of IFS on the set $\ov{V} \subset \R^d$ obtained by projection
of $\Sig \cap \{x\in \R^{d+1}: x_{d+1}=1\}$ onto $\R^d$. Notice that the cone $\Sig$ is strictly preserved by all $\Ak_t,\ t\ge 0$, by construction, hence by Lemma~\ref{lem-Hilbert}, these IFS are all
uniformly hyperbolic. 
 Both the IFS and their dependence on $t$ are real-analytic, since the IFS are given by rational functions. Condition (\ref{gluk1}) holds for $t\in [0,M]$, for any $M<\infty$, by uniform hyperbolicity and compactness.  Finally, observe that, given $\eps>0$, for $t$ sufficiently large, we have
$$
\pi_q(f_{i,t}(\ov{V})) \subset B_\eps([w_i]_q).
$$
 By construction, $[w_i]_q$ are all distinct, hence Corollary~\ref{cor_sep} applies for $\eps>0$ sufficiently small.  We obtain that for all $t\in [0,\infty)$ outside a set of Hausdorff dimension zero, the IFS $\Fk_t$ satisfies the SESDC condition on a non-empty set, and then Proposition~\ref{prop-Dioph2} implies that the $m$-tuple of matrices $(A_{i,t})_{i\in \Lam}$ is Diophantine for all $t$ outside of a zero-dimensional set, so certainly for Lebesgue-a.e.\ $t$. Now Lemma~\ref{Fubini} yields the desired claim.

\smallskip

(ii) We consider $(SL_{d+1}(\R))^m$ as a codimension-$m$ submanifold of $(GL_{d+1}(\R))^m \subset \R^{(d+1)^2m}$. In the proof of part (i) we showed that for a.e.\ $(A_i)_{i\in \Lam} \in  \mathcal{X}_{\Sig, m}$, the induced IFS on a subset of $\R^d$ satisfies the SESDC condition on a non-empty set. 
Suppose that there is a positive measure subset $\Ek\subset \Yk_{\Sig,m}$ for which the strong Diophantine condition is violated. 
Then for every $(A_i)_{i\in \Lam}\in \Ek$, the induced IFS $\Phi$ does not satisfy the SESDC on a non-empty set by Proposition~\ref{prop-Dioph2}. However, 
$(A_i)_{i\in \Lam}\in \Yk_{\Sig,m}$ and $(c_i A_i)_{i\in \Lam}\in \Xk_{\Sig,m}$, for any $c_i>0$, induce the same IFS on the projective space, and we get a set of positive measure in $\Xk_{\Sig,m}$ for which the  SESDC condition on a non-empty set does not hold. This is a contradiction, and the theorem is proved completely.
\end{proof}


\section{Dimension of the attractor}


Let $A \in SL_2(\mathbb{R})$ be a hyperbolic matrix. Then
$A^{*} A$ has distinct eigenvalues $\| A \|^2 > \| A \|^{-2}$. 
Let $(\cos t_A, \sin t_A)^{\mathrm{t}}$ be the unit eigenvector 
corresponding to the eigenvalue $\| A \|^{-2}$, where $t_A \in [0, \pi)$. 
We recall some basic properties of the map $\varphi_A$, the induced action of $A$ on $\RP^1\cong [0,\pi)$. 
For more details see sections 2.2, 2.3 and 2.4 in \cite{HS2017}.  
Below we use the Euclidean metric on $[0,\pi)$ and denote by $|F|$ the Lebesgue measure of a measurable $F\subset [0,\pi)$.
The following simple lemma is \cite[Section 2.4]{HS2017}. 

\begin{lem}\label{trivial0}
Let $A \in SL_2(\mathbb{R})$. Then $ \| A \|^{-2} \leq | \varphi_{A}'(x) | \le\| A \|^{2}$  for all $x\in [0,\pi)$.
Furthermore, for any $\eps > 0$ there exists $C_\eps > 1$ such that 
$ | \varphi_{A}'(x) | \le C_\eps \| A \|^{-2}$ 
for all $x \in [0, \pi) \setminus (t_A - \eps, t_A + \eps)$. 
\end{lem}



The following lemma is now immediate.
 
\begin{lem}\label{trivial}
Let $U \subset [0, \pi)$ be an open set, with $|U| < \pi$.
Then, for every $\eps > 0$ there exists $C_\eps > 1$ such that for 
any $A \in SL_2( \mathbb{R} )$ with $(t_A-\eps,t_A+\eps) \subset U$, we have 
\begin{equation*}
\pi - C_\eps \| A \|^{-2} < | \varphi_A( U ) | < \pi.  
\end{equation*}
\end{lem} 



Let $\mathcal{A} = \{A_i \}_{i \in \Lambda}$ be a finite collection of $SL_2(\mathbb{R})$ matrices and let $\Phi = \{\varphi_A\}_{A\in \Ak}$ be the corresponding IFS on $[0,\pi)\cong\RP^1$.
 We continue to use the notation: $$\Phi(E) = \bigcup_{A\in \Ak} \varphi_A(E).$$
Recall that a strictly invariant multicone $U\subset [0,\pi)$ is  a nonempty open set having finitely many connected components with disjoint closures, such that $\ov{U}\ne \RP^1$ and 
$\Phi(\ov{U})\subset U$. 
By Theorem~\ref{th-unihyp}, the set $\A$ is hyperbolic, which means that there exist $c>0$ and $\lam>1$ such that
\begin{equation} \label{eq-hyp}
\|A_\bi\| \ge c\lam^n\ \ \mbox{for all}\ \bi\in \Lam^n,\ n\in \N.
\end{equation}


\begin{lem} \label{lem-contract} Let $U$ be a strictly invariant multicone for the IFS $\Phi$. Then there exists a constant $C_1>1$  such that

{\rm (i)} we have
\begin{equation}\label{eq-BDP}
\|A_\bi\|^{-2} \le |\varphi'_{\bi}(x)| \le C_1 \|A_\bi\|^{-2},\ \ \mbox{for all}\ \ x\in \ov{U},\ \bi\in \Lam^n,\ n\in \N;
\end{equation}

{\rm (ii)} the Bounded Distortion Property holds for $\Phi$ on $U$: 
\begin{equation} \label{eq-bdp}
\frac{1}{C_1} \le \frac{|\varphi_{\bi}'(x)|}{|\varphi_{\bi}'(y)|} \le C_1\ \ \mbox{for all}\ \ x,y\in \ov{U}, \ \bi\in \Lam^n,\ n\in \N;
\end{equation}

{\rm (iii)} we have
\begin{equation} \label{eq-HYP}
r_1^n \le |\varphi_{\bi}'(x)|\le C_2\,\lam^{-2n}, \ \mbox{for all}\ \ x\in \ov{U},\ \bi\in \Lam^n,\ n\in \N,
\end{equation}
where $C_2 = C_1/c^2$, with $c>0$ and $\lam>1$  from (\ref{eq-hyp}), and $r_1 = \left(\max_{A\in \Ak}\|A\|\right)^{-2}$;

{\rm (iv)} we have
$$
s = s_\Ak/2,
$$
where $s$ is the unique solution of  Bowen's equation $P_\Phi(s)=0$, with the pressure given by (\ref{Bowen0}) and $s_\Ak$ is the critical exponent, given by (\ref{critic}).
\end{lem}

\begin{proof}
(i) In view of Lemma~\ref{trivial0}, we only need to check the right inequality in (\ref{eq-BDP}). By Lemma~\ref{trivial} and strict invariance of $U$, we have for every $A_\bi$, with $ \bi\in \Lam^n$ and $n$ sufficiently large, that $t_{A_\bi}\not\in U$.  
Since $\Phi(\ov U) \subset U$, there exists $\eps>0$ such that 
$$
(t_{A_\bi}-\eps, t_{A_\bi}+\eps) \cap \Phi(\ov{U}) = \emptyset,\ \ \mbox{for all}\ \bi\in \Lam^n,\ n\ge n_0.
$$
For every $\bi = i_1\ldots i_n$, with $n\ge n_0$, and every $x\in \ov{U}$,
\begin{eqnarray*}
|\varphi'_\bi(x)| & \le &  |\varphi'_{i_1\ldots i_{n-1}}(\varphi_{i_n}(x))|\cdot |\varphi_{i_n}'(x)|  \\
& \le  & C_\eps \|A_{i_1\ldots i_{n-1}}\|^{-2}\|A_{i_n}\|^2  \mbox{\ \ \ \ \ \ \ \ \  \ \ \ \ \ \ \ \ \ \ by Lemma~\ref{trivial0}}\\
& \le & C_\eps (\max_{A\in \Ak}\|A\|)^4 \|A_\bi\|^{-2}.
\end{eqnarray*}
This confirms (\ref{eq-BDP}) for $n\ge n_0$, with $C_1:= C_\eps (\max_{A\in \Ak}\|A\|)^4$, and for the first finitely many $n$ it trivially holds with some constant.

(ii) is  immediate from (i).  

(iii) the left inequality follows from Lemma~\ref{trivial0}, and the right is a consequence of (i) and (\ref{eq-hyp});

(iv) easily follows from (\ref{Bowen0}) and (\ref{critic}).
\end{proof}

Let $p = ( p_i )_{i \in \Lambda}$ be a probability vector, and let $x_0 \in U$. 
Let $\chi_{\Phi, p}$ be the almost sure value of the limit 
\begin{equation}\label{lim}
\lim_{n \to \infty} - \frac{1}{n} \log |  ( \varphi_{ i_1 \cdots i_n } )' (x_0)  |, 
\end{equation}
where $i_1, i_2, \cdots \in \Lambda$ is a sequence chosen  
randomly according to the probability vector $p = ( p_i )_{ i \in \Lambda}$. 
The equations (\ref{eq-BDP}) and (\ref{Lyap1}) imply 
\be \label{eq-Lyap}
\chi_{ \Phi, p } = 2 \chi_{ \mathcal{A}, p }.
\ee

Consider the symbolic space $\Lam^\N$ with the measure $\mu=p^\N$,  and the shift transformation $\sig: \bi = i_1 i_2 i_3 \ldots \mapsto i_2 i_3 \ldots$, which is $\mu$-invariant and ergodic. The natural projection $\Pi:\Lam^\N\to \RP^1$ is defined by
$$
\Pi(\bi) = \lim_{n\to \infty} \varphi_{i_1\ldots i_n} (x_0),\ \ x_0\in \ov U,
$$
where the limit exists and is independent of $x_0$ by contraction properties of the IFS. Observe that
\be \label{eq-pro}
\varphi_{i_1} (\Pi(\sig\bi) ) = \Pi(\bi)\ \ \mbox{for all}\ \bi\in \Lam^\N.
\ee
Consider the function $h: \Lam^\N\to \R$ defined by
$$
h(\bi) = - \log|\varphi_{i_1}'(\Pi(\sig\bi))|,\ \bi \in \Lam^\N.
$$
Since $|\varphi_A'|$ is bounded away from $0$ and $\infty$ on $\ov U$, we have that $h$ is integrable with respect to $\mu$. The following lemma is standard, but we include the proof for completeness.

\begin{lem} \label{lem-Lyap} We have
$$
\chi_{\Phi,p} = -\int_{\Lam^\N} \log\bigl|\varphi_{i_1}'(\Pi(\sig\bi))\bigr|\,d\mu(\bi).
$$
\end{lem}

\begin{proof}
In view of the bounded distortion (\ref{eq-bdp}), for $\mu$-a.e.\ $\bi$,
\begin{eqnarray*}
\chi_{\Phi,p} & = & \lim_{n \to \infty} - \frac{1}{n} \log |  ( \varphi_{ i_1 \cdots i_n } )' (x_0)  | \\
& = & \lim_{n \to \infty} - \frac{1}{n} \log \left|  ( \varphi_{ i_1 \cdots i_n } )^{'} ( \Pi (\sig^n \bi)) \right| \\
& = & \lim_{n\to \infty} - \frac{1}{n}  \sum_{k=0}^{n-1} h(\sig^k \bi).
\end{eqnarray*}
using the chain rule and (\ref{eq-pro}) in the last step. The proof is finished by an application of the Birkhoff Ergodic Theorem.
\end{proof}

In order to apply Theorem~\ref{thmHS} we need the assumption of total irreducibility. It is achieved with the help of the following lemma.

\begin{lem} \label{lem-irred}
Let $\A$ be a finite set of $SL_2(\R)$ matrices having a strictly invariant multicone $U$, and let $\Phi_\A$ be the associated IFS on $U$. If the attractor of $\Phi_\A$ is not a singleton (equivalently, not all attracting fixed points of $\varphi_A$, $A\in \A$, coincide), then either

{\bf (a)}  the group generated by $\Ak$ is totally irreducible, or 

{\bf (b)} $\Ak$ is conjugate in $SL_2(\R)$ to a finite subset of upper triangular matrices with the upper-left entry of absolute value less than one. In the latter case the IFS $\Phi_\A$ is conjugate, via a linear-fractional transformation, to a linear contracting IFS (an IFS of contracting affine linear maps) on $\R$.
\end{lem}

\begin{proof}
By assumption, the attractor of $\Phi_\A$ is not a singleton, hence it contains at least two distinct fixed points of some $\varphi_{A_1}$ and $\varphi_{A_2}$, and then it is perfect.
%
Suppose that the group generated by $\Ak$ is not totally irreducible. Then there is a finite invariant subset $F$ for $\Phi_\A$. If $F$ intersects $\ov{U}$, we get a contradiction, since the forward orbit of any point in $\ov{U}$ under the semigroup generated by $\Phi_\A$ is dense in the attractor. Thus $F$ is contained in the complement $\ov{U}^c$. However, notice that the IFS $\Phi_{\A^{-1}}$, generated by the inverses of the matrices from $\Ak$, is contracting on $U^c$ (this is because $f(\ov{U}) \subset U$ implies $f^{-1}(U^c) \subset \ov{U}^c$). By the previous argument, since $F$ is finite, it must be a singleton. Hence all the attracting points of $\varphi_{A^{-1}}$, that is, all the repelling points of $\varphi_A,\ A\in \A$, coincide. However, the latter implies that we can conjugate the set of matrices in such a way that the repelling point is $0\in [0,\pi)\approx \RP^1$. The resulting matrices are upper-triangular, with the upper-left entry of absolute value less than one. The corresponding IFS on $\R^*$ has the repelling point at $\infty$, which means that we get a contracting linear IFS on $\R$.
\end{proof}

\begin{proof}[Proof of Theorem \ref{thm-attr}] Recall that $\A$ is a finite set of $SL_2(\R)$ matrices satisfying the strong Diophantine condition and having a strictly invariant multicone $U$, and $\Phi=\Phi_\Ak$ is the associated IFS on $U$. Then $\Phi$ has a compact attractor
$K$ (not a singleton by assumption), and our goal is to show that $\dim_H(K) =\min\{1,s\}$, where $s$ is the unique solution of Bowen's equation (\ref{Bowen0}). We have two cases to consider. Note that both the assumption and the conclusion of the theorem are invariant under conjugacy. If we are in the case (b) of Lemma~\ref{lem-irred}, then applying Hochman's Theorem \cite[Cor 1.2]{Hochman2014} on the dimension of self-similar sets in $\R$, yields the result. 
One only needs to note that the strong Diophantine condition on the set of $SL_2(\R)$ matrices described in case (b) 
implies the strong exponential separation condition on a finite subset of $\R$. We postpone the proof of this fact to Lemma~\ref{lem-HS} in the Appendix.
In the case (a) of Lemma~\ref{lem-irred}, the group generated by $\Ak$ is totally irreducible, and we will assume this for the rest of the proof. 

It is known that 
\begin{equation}\label{dim_ineq}
\dim_H (K) \leq \min\{1,s\},
\end{equation}
see the appendix for a  short proof. 
Let us show the opposite inequality.  

Let $d_n > 0$ be the solution of the equation 
\begin{equation*}
\sum_{\bi\in \Lam^n} | U_\bi |^{d_n} = 1, 
\end{equation*}
where $U_\bi =\varphi_\bi(U)$.
It is not hard to see that
\begin{equation}\label{dim_lim}
\lim_{n \to \infty} d_n = s.
\end{equation}
For convenience of the reader, we include the proof in the appendix, following  \cite{SiSo1999}.

Let $p^{(n)} = ( p^{(n)}_\bi )_{\bi \in \Lam^n}$ be the probability vector 
such that $p^{(n)}_\bi = | U_\bi |^{d_n}$. Denote
$$
\Phi^n = \{\varphi_\bi:\ \bi\in \Lam^n\},\ n\in \N.
$$
Let $\eta^{(n)}$ be the invariant probability measure for the IFS $\Phi^n$ on $\ov U$, corresponding to $p^{(n)}$. 
Since $\eta^{(n)}$ is supported on $K$, 
we have $\dim \eta^{(n)} \leq \dim_{H} (K)$. 

Note that $\Ak$ satisfies the assumptions of Theorem~\ref{thmHS}. Indeed, the existence of a strictly invariant multicone  implies that all the matrices in $\Ak$ are hyperbolic, hence the group generated by $\Ak$ is unbounded, and we are now under the total irreducibility assumption.
Thus the Furstenberg measure for 
$(\Ak^n,p^{(n)})$ is unique, and it coincides with $\eta^{(n)}$.  Since $\Ak$ is strongly Diophantine, we have that $\Ak^n$ is strongly Diophantine as well.  
Now, by Theorem \ref{thmHS} and (\ref{eq-Lyap}) we have 
\begin{equation*}
\min\Bigl\{1,\frac{ H( p^{(n)} ) }{ \chi_{ \Phi^n, p^{(n)} } }\Bigr\} \leq \dim_{H} (K). 
\end{equation*}

We claim that there exists $C > 0$ such that 
\begin{equation} \label{claim2}
\chi_{\Phi^n, p^{(n)} } \le -\sum_{ \bi \in \Lambda^n } | U_\bi |^{d_n} \log | U_\bi | + C\ \ \mbox{for all}\ \ n\in \N.  
\end{equation}
Indeed, by Lemma~\ref{lem-Lyap}, 
$$
\chi_{\Phi^n, p^{(n)} } \le \sum_{\bi \in\Lam^n} \mu^{(n)}([\bi])\cdot \log\Bigl[\bigl(\min_{x\in \ov{U}}|\varphi_\bi'(x)|\bigr)^{-1}\Bigr],
$$
where $[\bi]$ is the cylinder set of sequences starting with $\bi$. 
Here $\mu^{(n)}$ is the Bernoulli measure on $(\Lam^n)^\N$, with $\mu^{(n)}([\bi]) = |U_\bi |^{d_n}$ for $\bi \in \Lam^n$.
We have
\be \label{note1}
|U_\bi | \le |U|\cdot \max_{x\in \ov{U}}|\varphi_\bi'(x)| \le C_1 |U|\cdot \min_{x\in \ov{U}}|\varphi_\bi'(x)|,
\ee
by the Bounded Distortion Property (\ref{eq-bdp}).
Therefore,
\begin{equation*}
\begin{aligned}
\chi_{ \Phi^n, p^{(n)} }  
& \le  \sum_{ \bi \in \Lambda^n } | U_\bi |^{d_n} \log \Bigl(\frac{ C_1 | U | }{ | U_\bi | } \Bigr) \\
&= -\sum_{\bi \in \Lambda^n} | U_\bi |^{d_n} \log | U_\bi | + \log (C_1 | U |),
\end{aligned}
\end{equation*}
confirming (\ref{claim2}).
Now we can estimate 
\begin{equation*}
\begin{aligned}
\frac{ H( p^{(n)} ) }{ \chi_{ \Phi^n, p^{(n)} } } & \ge  
\frac{ -\sum_{ \bi \in \Lambda^n } | U_\bi |^{d_n} \log\left( | U_\bi |^{d_n}\right)}
{ -\sum_{ \bi \in \Lambda^n } | U_\bi |^{d_n} \log | U_\bi | + C } \\
&= d_n \left( 1 + \frac{C}{ -\sum_{ \bi \in \Lambda^n } | U_\bi |^{d_n} \log | U_\bi | } \right)^{-1}. 
\end{aligned}
\end{equation*}
Since $\lim_{n \to \infty} d_n = s$ and 
\begin{eqnarray*}
-\sum_{ \bi \in \Lambda^n } | U_\bi |^{d_n} \log | U_\bi | & \ge & \sum_{ \bi \in \Lambda^n } | U_\bi |^{d_n} \cdot \left(-\log|U| - \log C_2 + 2n\log \lam\right) \\
& = & -\log|U| - \log C_2 + 2n\log \lam \to \infty,\ n\to \infty,
\end{eqnarray*}
by (\ref{note1}) and Lemma~\ref{lem-contract}(iii),
we obtain $\min\{1,s\} \leq \dim_{H} (K)$, as desired. Finally, $s=s_\Ak/2$ by Lemma~\ref{lem-contract}(iv).
\end{proof}


\section{Appendix: miscellaneous proofs} \label{appendix}
 
\subsection{Proof of (\ref{dim_lim}) \cite{SiSo1999}} We have a projective IFS $\Phi = \{\varphi_i\}_{i\in \Lam}$ on a strictly invariant multicone $U$.
Observe that
$$
P_\Phi(t) = \lim_{n\to \infty} \frac{1}{n} \log \sum_{\bi \in \Lam^n} \|\varphi_\bi'\|^t = \lim_{n \to \infty} \frac{1}{n} \log \sum_{\bi \in \Lambda^n} \left| U_{\bi} \right|^t,
$$
 by the Bounded Distortion Property (\ref{eq-bdp}). Let
\begin{equation*}
Q_n = \frac{1}{n} \log \sum_{\bi \in \Lambda^n} \left| U_{\bi} \right|^s. 
\end{equation*}
Since $P_\Phi(s ) = 0$, we have $\lim_{n \to \infty} Q_n = 0$. 
Recall (\ref{eq-HYP}), which says that
$$
r_1^n \le {\|\varphi'_\bi\|}_{\ov{U}} \le C_2 \lam^{-2n},
$$
for $r_1\in (0,1), C'_2>0$, and $\lam>1$.
Then 
$r^n_1 |U| \le  |U_\bi| \le C_2\lam^{-2n} |U|$ for $\bi\in\Lam^n$, and hence we have
$$
|U_\bi|^s \in |U_\bi|^{d_n} \cdot \left[(r_1^n |U|)^{s-d_n}, (C_2 \lam^{-2n})^{s-d_n}\right],
$$
where we use the convention  that $[a,b]=[b,a]$ if $a>b$.
In view of  $\sum_{\bi\in \Lambda^n} |U_\bi|^{d_n} = 1$, we have
$$
Q_n \in (d_n-s)\cdot \left[2\log\lam - \frac{\log C_2}{n}\,,\, -\log r_1 -\frac{\log|U|}{n}\right],
$$
whence for $n$ sufficiently large,
$$
d_n - s \in Q_n \cdot  \left[\left(-\log r_1 - \frac{\log|U|}{n}\right)^{-1}, \left(2\log \lam -\frac{\log C_2}{n}\right)^{-1} \right],
$$
which implies $d_n\to s$, as desired. \qed

\subsection{Proof of (\ref{dim_ineq})}
Fix $\eps > 0$. Then for sufficiently large $n$ we have 
$d_n < s + \eps/2$, and
\begin{equation*}
\sum_{\bi \in \Lambda^n} | U_\bi |^{s + \eps} \le \sum_{\bi \in \Lambda^n} | U_\bi |^{d_n + \eps/2} 
\le  (C_2\lam^{-2n} |U| )^{\eps/2} \to 0, \text{\ as } n \to \infty. 
\end{equation*}
Therefore, the $(s + \eps)$-dimensional Hausdorff measure of $K$ is zero. 
By the definition of the Hausdorff dimension, this proves (\ref{dim_ineq}). \qed

\subsection{From the Diophantine property to exponential separation}

\begin{lem} \label{lem-HS}
Suppose that $\A$ is a finite strongly Diophantine set of upper-triangular $SL_2(\R)$ matrices with the upper-left entry of absolute value less than one. Then the associated IFS satisfies the strong exponential separation condition on the set $\{0,1\}$.
\end{lem}

Instead of $\{0,1\}$ one can take any two distinct points in $\R$.
We start with the following general elementary claim.

\medskip

\noindent {\bf Claim.} {\em
Let $G$ be a matrix group, and  $\Ak=\{A_1,\ldots,A_m\} \subset G$ is a finite strongly Diophantine set. Suppose that $\widetilde{\Ak}=\{\wt A_1,\ldots, \wt A_m\}$ is such that $\wt A_j = \pm A_j$ for all $j\le m$. Then $\wt \Ak$ is strongly Diophantine as well.}

\begin{proof}[Proof of the claim] We argue by contradiction and assume that for any $\eps>0$ there exist $n\in \N$ and $\bi\ne \bj$ in $\Lam^n$  such that
$$\|\wt A_\bi - \wt A_\bj\| = \|A_\bi \pm A_\bj\| \le \eps^n.
$$
For $\eps>0$ sufficiently small, $\|A_\bi - A_\bj\| \le \eps^n$ is impossible, thus $\|A_\bi + A_\bj\|\le \eps^n$.
But then we have
\begin{eqnarray*}
\|A_{\bi\bi} - A_{\bj\bj}\| = \|A_\bi^2 - A_\bj^2\| & = & \|(A_\bi + A_\bj)A_\bi - A_\bj(A_\bi + A_\bj)\| \\
& \le & \eps^n(\|A_\bi\| + \|A_\bj\|) \le 2C^n \eps^n,
\end{eqnarray*}
where $C= \max_{A\in \Ak} \|A\|$. This contradicts the strong Diophantine property of $\Ak$ for $\eps>0$ sufficiently small.
\end{proof}

\begin{proof}[Proof of the lemma]
Observe that $A$ and $-A$ act identically on $\RP^1$. 
In view of the Claim, we can assume that the matrices of $\Ak$ have the form
$$
A_j = \left[ \begin{array}{cc} \sqrt{\lam_j} & a_j/\sqrt{\lam_j} \\[1.2ex] 0 & 1/\sqrt{\lam_j} \end{array} \right],\ \ \mbox{with}\ \ \lam_j \in (0,1),\ a_j\in \R.
$$
All matrices of the semigroup generated by $\Ak$ have a similar upper-triangular form: $$A_\bj = \left[ \begin{array}{cc} \sqrt{\lam_\bj} & a_\bj/\sqrt{\lam_\bj} \\[1.2ex] 0 & 1/\sqrt{\lam_\bj} \end{array} \right],\ \ \mbox{with}\ \ \bj \in \Lam^n,\ \lam_\bj = \lam_{j_1}\cdots  \lam_{j_n},\ a_\bj \in \R.
$$
The matrices of $\Ak$ act on $\RP^1 \approx \R^*$ by $A_j \cdot x = \lam_j x + a_j$. Assume that $\Phi_\Ak$ does not satisfy the strong exponential separation condition on $\{0,1\}$. Then for any $\eps>0$ there exist $n\in \N$ and $\bi \ne \bj$ in $\Lam^n$ such that $|a_\bi - a_\bj| \le \eps^n$ and 
$$
|(\lam_\bi + a_\bi) - (\lam_\bj+ a_\bj)| \le \eps^n \implies |\lam_\bi - \lam_\bj| \le 2\eps^n.
$$
It follows that
$$
|\sqrt{\lam_\bi} - \sqrt{\lam_\bj}|  = \frac{|\lam_\bi - \lam_\bj|}{\sqrt{\lam_\bi} + \sqrt{\lam_\bj}} \le \frac{2\eps^n}{2\min_{\bj\in \Lam^n} \sqrt{\lam_j}} \le \Bigl(\frac{\eps}{\min_{j\in \Lam}\sqrt{\lam_j}}\Bigr)^n.
$$
In a similar elementary way one can estimate from above the expressions $|1/\sqrt{\lam_\bi} - 1/\sqrt{\lam_\bj}|$ and $|a_\bi\sqrt{\lam_\bi} - a_\bj/\sqrt{\lam_\bj}|$ to conclude
$$
\|A_\bi - A_\bj\|\le (C'\eps)^n,
$$
for some $C'>0$, depending only on $\Ak = \{A_j\}_{j\in \Lam}$. This contradicts the strong Diophantine property of $\Ak$.
\end{proof}

\section*{Acknowledgements}
Thanks to  Balazs Bara\'ny for helpful discussions and for telling us about the papers \cite{ABY2010,BG2009}. We are deeply grateful to the referees for many helpful comments and suggestions, as well as pointing out several gaps in the original proofs.

\bibliographystyle{plain}
\bibliography{Solomyak_Takahashi}

\end{document}